\documentclass[11pt]{amsart}
\usepackage{graphicx,subcaption}
\usepackage{amsmath} 
\usepackage{amsthm,amsfonts,amssymb,mathrsfs,amscd,amstext,amsbsy} 
\usepackage{epic,eepic} 
\usepackage{yfonts}
\usepackage{paralist,enumerate}
\usepackage[all]{xy}
\usepackage{hyperref}
\hypersetup{colorlinks}

\usepackage{tikz}
\usetikzlibrary {positioning}

\newtheorem{theorem}{Theorem}[section]
\newtheorem{cor}[theorem]{Corollary}
\newtheorem{lemma}[theorem]{Lemma}

\newtheorem{theo}[theorem]{Theorem}

\newtheorem{rem}[theorem]{Remark}
\newtheorem{exa}[theorem]{Example}

\newtheorem{definition}[theorem]{Definition}
\newtheorem*{Definition*}{Definition}
\numberwithin{equation}{section}

\def\qed{\hfill \ifhmode\unskip\nobreak\fi\quad\ifmmode\Box\else$\Box$\fi\\ }

\begin{document}

\title[6-dimensional oriented $S^1$-manifold with 4 fixed points]{Circle actions on 6-dimensional oriented manifolds with 4 fixed points}
\author{Donghoon Jang}
\thanks{Donghoon Jang was supported by the National Research Foundation of Korea(NRF) grant funded by the Korea government(MSIT) (2021R1C1C1004158).}
\address{Department of Mathematics, Pusan National University, Pusan, Korea}
\email{donghoonjang@pusan.ac.kr}

\begin{abstract}
In this paper, we classify the fixed point data (weights and signs at the fixed points), of a circle action on a 6-dimensional compact oriented manifold with 4 fixed points. We prove that it agrees with that of a disjoint union of rotations on two 6-spheres, or that of a linear action on $\mathbb{CP}^3$. The former case includes that of Petrie's exotic action on $\mathbb{CP}^3$.
\end{abstract}

\maketitle

\tableofcontents

\section{Introduction}

$\indent$

Let the circle group $S^1$ act on a compact connected oriented manifold $M$ with non-empty finite fixed point set. Because the dimension of $M$ has the same parity as its fixed point set, the dimension of $M$ is even. Let $p$ be an isolated fixed point. The tangent space $T_pM$ at $p$ decomposes into $n$ irreducibles $T_pM=\oplus_{i=1}^n L_i$, where $\dim M=2n$.
On each $L_i$, the circle group acts as multiplication by $g^{w_{pi}}$ for all $g \in S^1$, for some non-zero integer $w_{pi}$. We may choose an orientation of $L_i$ so that $w_{pi}$ is positive for all $i$. The positive integers $w_{p1},\cdots,w_{pn}$ are called \textbf{weights} at $p$. Let $\epsilon(p)=+1$ if the orientation on $M$ and the orientation on the representation space $\oplus_{i=1}^n L_i$ agree, and $\epsilon(p)=-1$ otherwise, and call it the \textbf{sign} of $p$. Define the \textbf{fixed point data} of $p$ to be an ordered pair $(\epsilon(p),\{w_{p1},\cdots,w_{pn}\})$ where $\{w_{p1},\cdots,w_{pn}\}$ is a multiset, and simplify it as $\{\epsilon(p),w_{p1},\cdots,w_{pn}\}$ by writing the sign first and the weights next. 
By the \textbf{fixed point data} of $M$ we mean a collection $\cup_{p \in M^{S^1}} \{\epsilon(p),w_{p1},\cdots,w_{pn}\}$ of the fixed point datum of the fixed points of $M$.

If the action on $M$ has exactly one fixed point, say $p$, then $M$ must be the point itself, that is, $M=\{p\}$; in particular $\dim M=0$. If the action on $M$ has exactly two fixed points, the fixed point data of $M$ agrees with that of a rotation of $S^{2n}$ (Theorem \ref{t28}); in particular such an $M$ exists in any even dimension. If there is an odd number of fixed points, the dimension of the manifold is a multiple of 4 (see Corollary \ref{c24}). Examples with 3 fixed points exist in dimension 4, 8, and 16; circle actions on the complex projective space $\mathbb{CP}^2$, the quaternionic projective space $\mathbb{HP}^2$, and the octonionic projective space $\mathbb{OP}^2$. On the other hand, three fixed points cannot exist in dimension 12 \cite{J5}.

Suppose that there exists a $2n$-dimensional $M$ with $k$ fixed points, where $n>1$. By taking an equivariant connected sum along free orbits of $M$ and another $2n$-dimensional compact connected oriented $S^1$-manifold that has no fixed points, we can construct infinitely many $2n$-dimensional compact connected oriented $S^1$-manifolds with $k$ fixed points. In dimension 2, the 2-sphere is the only compact conected oriented $S^1$-manifold with non-empty finite fixed point set.

While we can construct infinitely many compact connected oriented $S^1$-manifolds in prescribed dimension and with prescribed number of fixed points if there exists one, the number of types of fixed point data could be finite. This is true if $k \leq 2$ as noted above, and if $\dim M \leq 6$; see \cite{J2} and \cite{J6}.

Circle actions with 4 fixed points have been studied on different types of 6-dimensional compact manifolds. Ahara studied almost complex manifolds with Todd genus 1 and $\int_M c_1^3 \neq 0$ \cite{A}, Tolman studied Hamiltonian actions on symplectic manifolds \cite{T}, and the author studied almost complex manifolds \cite{J3}.

In this paper, we prove that if a 6-dimensional oriented $M$ has 4 fixed points, only two types of fixed point data occur; $S^6 \sqcup S^6$ type and $\mathbb{CP}^3$ type.

\begin{theorem} \label{t11}
Let the circle act on a 6-dimensional compact connected oriented manifold $M$ with 4 fixed points. Then one of the following holds for the fixed point data of $M$.
\begin{enumerate}
\item $\{+,a,b,c\},\{-,a,b,c\},\{+,d,e,f\},\{-,d,e,f\}$ for some positive integers $a$, $b$, $c$, $d$, $e$, and $f$.
\item $\{+,a,a+b,a+b+c\}, \{-,a,b,b+c\}, \{+,b,c,a+b\}, \{-,c,b+c,a+b+c\}$. for some positive integers $a$, $b$, and $c$.
\end{enumerate}
\end{theorem}

Note that in Theorem \ref{t11}, Case (2) with $a=c$ belongs to Case (1).
In Case (1) of Theorem \ref{t11}, the fixed point data is the same as the equivariant sum (or a disjoint union) of rotations on two 6-spheres; see Example \ref{e28}. In Case (2) of Theorem \ref{t11}, the fixed point data is the same as a standard linear action of $S^1$ on the complex projective space $\mathbb{CP}^3$; see Example \ref{e29}. On the other hand, in Example \ref{e210}, we show that blow up at a fixed point, of a rotation on $S^6$ with 2 fixed points also has the fixed point data of Case (2) of Theorem \ref{t11}.

We note that Case (1) of Theorem \ref{t11} includes the fixed point data of the exotic $S^1$-action on $\mathbb{CP}^3$ constructed by Petrie \cite{P}. The strong Petrie conjecture asserts that if a compact oriented manifold $M$ is homotopy equivalent to a complex projective space $\mathbb{CP}^n$ and admits a non-trivial $S^1$-action, the total Pontryagin class of $M$ agrees with that of $\mathbb{CP}^n$ \cite{P}. The weak Petrie conjecture adds an assumption that the fixed point set is discrete. If $\dim M=6$, $M$ is a homotopy $\mathbb{CP}^3$, and the fixed point set is discrete, then $M$ has 4 fixed points. For the proof of weak Petrie conjecture in dimension 6, see \cite{D} and \cite{M1}.

The Petrie's example is exotic in a sense that the weights at the 4 fixed points are distinct from those of the linear action. For the exotic action, the (complex) weights at the fixed points are
\begin{center}
$\{7,2,3\}, \{-7,2,3\}, \{5,2,3\}, \{-5,2,3\}$,
\end{center}
respectively. The fixed point data of this action on $\mathbb{CP}^3$ as an oriented manifold is
\begin{center}
$\{+,7,2,3\}, \{-,7,2,3\}, \{+,5,2,3\}, \{-,5,2,3\}$.
\end{center}
Therefore, this fixed point data belongs to Case (1) of Theorem \ref{t11}.

We note that the two types of fixed point data of Theorem \ref{t11} include those of both the exotic action (as Case (1)) and the standard action (as Case (2)) on $\mathbb{CP}^3$, without assumption that a given manifold is homotopy equivalent to $\mathbb{CP}^3$.

The paper is organized as follows. In Section \ref{s2}, we provide background and preliminaries. In Section \ref{s3}, we discuss how to associate a multigraph to an oriented $S^1$-manifold $M$ with a discrete fixed point set. The vertex set of the multigraph is the fixed point set, and roughly speaking, if two fixed points $p$ and $q$ have weights $w$, we draw an edge between $p$ and $q$ with label $w$. The multigraph we assign is slightly different from previous work, but both encode the fixed point data while the former simplifies the proof of Theorem \ref{t11}. In Section \ref{s4}, we show that if $\dim M=6$ and $|M^{S^1}|=4$, there is a few number of possible multigraphs that encode the fixed point data of $M$. In Section \ref{s5}, for each possible multigraph, we determine the fixed point data of $M$, proving Theorem \ref{t11}. In Section \ref{s6}, we compare our result for oriented manifolds with results for different types of manifolds mentioned above. In Section \ref{s7}, we discuss how we convert the fixed point data of $M$ into the empty collection.

\section{Background and preliminaries} \label{s2}

In this section, we review background and properties of a circle action on a compact oriented manifold with a discrete fixed point set.

For an action of a group $G$ on a manifold $M$, denote by $M^G$ the fixed point set. That is,
\begin{center}
$M^G=\{m \in M \, | \, g \cdot m=m, \forall g \in G\}$.
\end{center}

Let the circle group act on a compact oriented manifold $M$. The \textbf{equivariant cohomology} of $M$ is
\begin{center}
$H_{S^1}^*(M):=H^*(M \times_{S^1} S^\infty)$.
\end{center}
The projection map onto the second factor $\pi:M \times_{S^1} S^\infty \to \mathbb{CP}^\infty$ gives rise to a push-forward map
\begin{center}
$\displaystyle \int_M := \pi_*:H_{S^1}^i(M) \to H^{i-\dim M}(\mathbb{CP}^\infty)$ for all $i$.
\end{center}
The Atiyah-Bott-Berline-Vergne localization theorem allows us to compute the push-forward map in terms of the fixed point data.

\begin{theorem} \emph{[\textbf{Atiyah-Bott-Berline-Vergne localization theorem}]} \cite{AB} \label{t21}
Let the circle act on a compact oriented manifold $M$. Given $\alpha \in H_{S^1}^*(M;\mathbb{Q})$,
\begin{center}
$\displaystyle \int_M \alpha=\sum_{F \subset M^{S^1}} \int_F \frac{\alpha|_F}{e_{S^1}(N_F)}$,
\end{center}
where the sum is taken over all fixed components, and $e_{S^1}(N_F)$ is the equivariant Euler class of the normal bundle to $F$.
\end{theorem}

For a compact oriented manifold $M$, the L-genus is the genus belonging to the power series $\frac{\sqrt{z}}{\sinh \sqrt{z}}$. The Atiyah-Singer index theorem proves that the L-genus of $M$ is equal to the index of the signature operator on $M$ \cite{AS}. The equivariant index of the signature operator on a compact oriented manifold $M$ equipped with a circle action, is independent of the choice of an element of $S^1$ and is equal to the signature of $M$. As a result, the following formula holds.

\begin{theo} \emph{[\textbf{Atiyah-Singer index theorem}]} \cite{AS} \label{t22} Let the circle act on a $2n$-dimensional compact oriented manifold $M$ with a discrete fixed point set. Then the signature of $M$ is
\begin{center}
$\displaystyle\textrm{sign}(M) = \sum_{p \in M^{S^1}} \epsilon(p) \cdot \prod_{i=1}^{n} \frac{1+t^{w_{pi}}}{1-t^{w_{pi}}}$
\end{center}
for all indeterminates $t$, and is a constant. \end{theo}

Consider a circle action on a compact oriented manifold with a discrete fixed point set. For each positive integer $w$, the number of times the weight $w$ occurs over all the fixed points, counted with multiplicity at each fixed point, is even.

\begin{lemma} \label{l22} \cite{M1}, \cite{J2}
Let the circle act on a $2n$-dimensional compact oriented manifold $M$ with a discrete fixed point set. For any positive integer $w$,
\begin{center}
$|\{w_{pi} : w_{pi}=w, 1 \leq i \leq n, p \in M^{S^1}\}| \equiv 0 \mod 2$.
\end{center}
\end{lemma}

In particular, if $\dim M=2n$ and there are $k$ fixed points, the total number $nk$ of weights over all fixed points must be even. Thus the following hold.

\begin{cor} \label{c24} Let the circle act on a compact oriented manifold. If the number of fixed points is odd, then the dimension of the manifold is divisible by 4. \end{cor}

Consider a circle action on an oriented manifold $M$. For a positive integer $w$, the group $\mathbb{Z}_w$ acts on $M$ as a subgroup of $S^1$. H. Herrera and R. Herrera proved the orientability of the set $M^{\mathbb{Z}_w}$ of points in $M$ that are fixed by the $\mathbb{Z}_w$-action.

\begin{lemma} \cite{HH} \label{l25} Let the circle act on a $2n$-dimensional oriented manifold $M$. Consider $\mathbb{Z}_w \subset S^1$ and its corresponding action on $M$. If $w$ is odd then the fixed point set $M^{\mathbb{Z}_w}$ is orientable. If $w$ is even and a connected component $F$ of $M^{\mathbb{Z}_w}$ contains a fixed point of the $S^1$-action, then $F$ is orientable. \end{lemma}

To prove Theorem \ref{t11}, we introduce some terminologies. Let $S^1$ act on a $2n$-dimensional compact oriented manifold $M$ with a discrete fixed point set. Let $w$ be a positive integer. Consider a connected component $F$ of $M^{\mathbb{Z}_w}$ such that $F \cap M^{S^1} \neq \emptyset$. Then $F$ is orientable by Lemma \ref{l25}; choose an orientation of $F$. Then the normal bundle $NF$ of $F$ is orientable. Take an orientation on $NF$ so that the orientation of $T_qF \oplus N_qF$ agrees with the orientation of $T_qM$ for all $q \in F$. If $p \in F \cap M^{S^1}$, this gives orientations of $T_pF$ and $N_pF$; let $\alpha_F(p)$ and $\alpha_N(p)$ denote the orientations, respectively. In the decomposition of the tangent space at $p$ into $n$ irreducibles $T_pM=\oplus_{i=1}^n L_i$, for each $i$ we choose an orientation of $L_i$ so that $S^1$ acts on $L_i$ with a positive weight $w_{pi}>0$, $1 \leq i \leq n$. Rearrange $L_i$'s so that $T_pF=L_1 \oplus \cdots \oplus L_m$ and $N_pF=L_{m+1} \oplus \cdots \oplus L_n$.

\begin{definition} \label{d24}
\begin{enumerate}[(1)]
\item $\epsilon_F(p)=+1$ if the orientation on $T_pF$ and the orientation on $L_1 \oplus \cdots \oplus L_m$ (in which all $w_{pi}$ are positive) agree, and $\epsilon_F(p)=-1$ otherwise.
\item $\epsilon_N(p)=+1$ if the orientation on $N_pF$ and the orientation on $L_{m+1} \oplus \cdots \oplus L_n$ (in which all $w_{pi}$ are positive) agree, and $\epsilon_N(p)=-1$ otherwise. 
\item $\epsilon_M(p)=+1$ if the orientation on $T_pM$ and the orientation on $L_{1} \oplus \cdots \oplus L_n$ (in which all $w_{pi}$ are positive) agree, and $\epsilon_M(p)=-1$ otherwise. 
\end{enumerate} 
\end{definition}
By the definitions, $\epsilon(p)=\epsilon_M(p)$ and $\epsilon_M(p)=\epsilon_F(p) \cdot \epsilon_N(p)$. In \cite{J2}, the author proved relationships between the weights at the fixed points that lie in the same connected component of $M^{\mathbb{Z}_w}$.

\begin{lemma} \label{l26} \cite{J2}
Let the circle act on a compact oriented manifold $M$ with a discrete fixed point set. Let $w$ be a positive integer. Suppose that no multiples of $w$ occur as weights, other than $w$ itself. Then for any connected component $F$ of $M^{\mathbb{Z}_w}$, the number of $S^1$-fixed points $p$ in $F$ with $\epsilon_F(p)=+1$ and that with $\epsilon_F(p)=-1$ are equal. Moreover, we can pair points in $F \cap M^{S^1}$ such that
\begin{enumerate}
\item If $(p,q)$ is a pair, then $\epsilon_F(p)=-\epsilon_F(q)$.
\item If $\{w_{p1},\cdots,w_{pm}\}$ and $\{w_{q1},\cdots,w_{qm}\}$ are the weights of $N_pF$ and $N_qF$, there exist a bijection $\sigma_m:\{1,\cdots,m\}\rightarrow\{1,\cdots,m\}$ and a function $\nu:\{1,\cdots,m\} \to \{-1,1\}$ such that $w_{pi} \equiv \nu(i) \cdot w_{q \sigma_m(i)} \mod w$ for any $1 \leq i \leq m$. Moreover, $\epsilon_M(p)=\epsilon_M(q) \cdot (-1)^{\nu_{p,q}^-+1}$, where $\nu_{p,q}^-$ is the number of $i$'s with $\nu(i)=-1$.
\end{enumerate}
\end{lemma}

The following lemma can be proved easily using Theorem \ref{t22}.

\begin{lemma} \label{l28} \cite{J2}
Let the circle act on a compact oriented manifold $M$ with a discrete fixed point set. Suppose that every weight at any fixed point is equal to $w$ for some positive integer $w$. Then the number of fixed points $p$ with $\epsilon(p)=+1$ and the number of fixed points $p$ with $\epsilon(p)=-1$ are equal; in particular, there is an even number of fixed points. Moreover, $\textrm{sign}(M)=0$. \end{lemma}

We review known classification results on the fixed point data of an oriented $S^1$-manifold. If there are 2 fixed points, the fixed point data is the same as that of a rotation on an even dimensional sphere.

\begin{theorem} \label{t28} \cite{Ko2}
Let the circle act on a compact oriented manifold with two fixed points $p$ and $q$. Then the weights at $p$ and $q$ agree up to order and $\epsilon(p)=-\epsilon(q)$. \end{theorem}

In dimension 4, if a circle action has a discrete fixed point set, the fixed point data is classified.

\begin{theorem} \label{t70} \cite{J2} Let the circle act effectively on a 4-dimensional compact oriented manifold $M$ with a discrete fixed point set. Then the fixed point data of $M$ can be achieved in the following way: begin with the empty collection, and apply a combination of the following steps.
\begin{enumerate}
\item Add $\{+,a,b\}$ and $\{-,a,b\}$, where $a$ and $b$ are relatively prime positive integers.
\item Replace $\{+,c,d\}$ by $\{+,c,c+d\}$ and $\{+,d,c+d\}$.
\item Replace $\{-,e,f\}$ by $\{-,e,e+f\}$ and $\{-,f,e+f\}$.
\end{enumerate}\end{theorem}

In Theorem \ref{t11}, two types of the fixed point data occur for a circle action on a 6-dimensional compact oriented manifold with isolated fixed points. For each type, such a manifold exists.

\begin{exa} \label{e28}
Let the circle act on $S^6$ by
\begin{center}
$g \cdot (z_1,z_2,z_3,x)=(g^a z_1, g^b z_2, g^c z_3,x)$
\end{center}
for all $g \in S^1 \subset \mathbb{C}$, where $S^6=\{(z_1,z_2,z_3,x) \in \mathbb{C}^3 \times \mathbb{R} \ | \ \sum_{i=1}^3 |z_i|^2 + x^2=1\} \subset \mathbb{C}^3 \times \mathbb{R}$ and $a$, $b$, and $c$ are positive integers. The action has two fixed points $p_1=(0,0,0,1)$ and $p_2=(0,0,0,-1)$, and $\Sigma_{p_1}=\{+,a,b,c\}$ and $\Sigma_{p_2}=\{-,a,b,c\}$. Consider another $S^6$ with action
\begin{center}
$g \cdot (z_1,z_2,z_3,x)=(g^d z_1, g^e z_2, g^f z_3,x)$.
\end{center}
Taking an equivariant connected sum (or a disjoint union) of the two 6-spheres along free orbits of them, the fixed point data of Case (1) in Theorem \ref{t11} is achieved.
\end{exa}

A standard linear action on $\mathbb{CP}^3$ has the fixed point data of Case (2) of Theorem \ref{t11}.

\begin{exa} \label{e29}
Let the circle act on $\mathbb{CP}^3$ by
\begin{center}
$g \cdot [z_0:z_1:z_2:z_3]=[z_0:g^a z_1:g^{a+b} z_2:g^{a+b+c} z_3]$,
\end{center}
for all $g \in S^1 \subset \mathbb{C}$, for some positive integers $a$, $b$, and $c$. The action has 4 fixed points $[1:0:0:0]$, $[0:1:0:0]$, $[0:0:1:0]$, and $[0:0:0:1]$. As complex representations where the sign of every weight is well-defined, the weights at the fixed points are $\{a,a+b,a+b+c\}$, $\{-a,b,b+c\}$, $\{-a-b,-b,c\}$, and $\{-a-b-c,-b-c,-c\}$, respectively. As real representations, the fixed point data at each fixed point is $\{+,a,a+b,a+b+c\}$, $\{-,a,b,b+c\}$, $\{+,b,c,a+b\}$, and $\{-,c,b+c,a+b+c\}$, respectively. This fixed point data is the same as Case (2) in Theorem \ref{t11}.
\end{exa}

On the other hand, as explained in the introduction, blow up at a fixed point, of a rotation on $S^6$ with two fixed points also has the fixed point data of Case (2) in Theorem \ref{t11}.

Let $a$, $b$, and $c$ be mutually distinct non-zero integers. Let $S^1$ act on $\mathbb{C}^3$ by
\begin{center}
$g \cdot (z_1,z_2,z_3)=(g^{a} z_1, g^b z_2, g^{c} z_3)$
\end{center}
for all $g \in S^1$. The origin is fixed by this action, and the (complex) weights at the origin are $\{a,b,c\}$. Now we blow up the origin. Blowing up the origin replaces the origin with all complex straight lines through it. The blown up space $\widetilde{\mathbb{C}^3}$ can be described by
\begin{center}
$\widetilde{\mathbb{C}^3}=\{(z,\ell) \, | \, z \in \ell\} \subset \mathbb{C}^3 \times \mathbb{CP}^2$.
\end{center}
The $S^1$-action on $\mathbb{C}^3$ naturally extends to act on $\widetilde{\mathbb{C}^3}$ by
\begin{center}
$g \cdot ((z_1,z_2,z_3),[w_0:w_1:w_2])=((g^{a} z_1, g^b z_2, g^{c} z_3),[g^{a} w_0: g^b w_1:g^{c} w_2])$
\end{center}
for all $((z_1,z_2,z_3),[w_0:w_1:w_2]) \in \widetilde{\mathbb{C}^3}$. The action on $\widetilde{\mathbb{C}^3}$ has three fixed points $p_1=((0,0,0),[1:0:0])$, $p_2=((0,0,0),[0:1:0])$, and $p_3=((0,0,0),[0:0:1])$, that have (complex) weights $\{a,b-a,c-a\}$, $\{b,a-b,c-b\}$, and $\{c,a-c,b-c\}$, respectively.

\begin{exa} \label{e210}
Let $S^1$ act on $S^6$ by $g \cdot (z_1,z_2,z_3,x)=(g^c z_1,g^b z_2,g^{b+a} z_3,x)$ for some positive integers $a$, $b$, and $c$. The fixed point datum at $p_1=(0,0,0,1)$ and $p_2=(0,0,0,-1)$ are $\{+,c,b,a+b\}$ and $\{-,c,b,a+b\}$, respectivedly. Next, identify a neighborhood of $p_2$ with $\mathbb{C}^3$ so that the $S^1$-action near $p_2$ is identified with $g \cdot (z_1,z_2,z_3)=(g^{-c} z_1, g^b z_2, g^{a+b} z_3)$. Under this identification, we blow up the origin (the fixed point $p_2$) in $\mathbb{C}^3$ (a neighborhood of $p_2$) equivariantly to replace the origin by the set of complex straight lines through it; doing it equivariantly yields instead of one fixed point the origin (that has complex $S^1$-weights $\{-c,b,a+b\}$), three fixed points $q_1$, $q_2$, $q_3$, whose weights as complex $S^1$-representations are $\{-c,b+c,a+b+c\}$, $\{-b-c,a,b\}$, and $\{-a-b-c,-a,a+b\}$, respectively. The blown up manifold $\widetilde{S^6}$ is then equipped with a circle action having 4 fixed points $p_1$, $q_1$, $q_2$, and $q_3$ that have the fixed point data as real $S^1$-representations $\{+,b,c,a+b\}$, $\{-,c,b+,c,a+b+c\}$, $\{-,a,b,b+c\}$, and $\{+,a,a+b,a+b+c\}$. This fixed point data also agrees with Case (2) in Theorem \ref{t11}.
\end{exa}

\section{Special multigraph} \label{s3}

It has been already known that for a compact oriented $S^1$-manifold with a discrete fixed point set, we can associate a labeled multigraph, where the vertex set is the fixed point set and edges are drawn in terms of weights \cite{J2}, \cite{M1}. To simplify the proof of our main result, we consider a special type of multigraph by requiring an additional property.

\begin{definition} \label{d31} A \textbf{multigraph} $\Gamma$ is an ordered pair $\Gamma=(V,E)$ where $V$ is a set of vertices and $E$ is a multiset of unordered pairs of vertices, called \textbf{edges}. A multigraph is called \textbf{signed} if every vertex is assigned a number $+1$ or $-1$. A multigraph is called \textbf{labeled}, if every edge $e$ is labeled by a positive integer $w(e)$, called the \textbf{label}, or the \textbf{weight} of the edge. Alternatively, a multigraph is called labeled if there is a map from $E$ to the set of positive integers. Let $\Gamma$ be a labeled multigraph. The \textbf{weights} at a vertex $v$ is a multiset that consists of labels (weights) $w(e)$ for each edge $e$ at $v$. A multigraph $\Gamma$ is called \textbf{$n$-regular}, if every vertex has $n$-edges. \end{definition}

\begin{definition} \label{d32} Let the circle act on a compact oriented manifold $M$ with a discrete fixed point set. We say that a signed labeled multigraph $\Gamma$ \textbf{describes} (the fixed point data) of $M$, if the following hold.
\begin{enumerate}
\item The vertex set of $\Gamma$ is the fixed point set $M^{S^1}$.
\item For every fixed point $p$, the weights (labels) of the edges at the vertex $p$ are the weights at the fixed point $p$. 
\item For every fixed point $p$, the sign of the vertex $p$ is equal to the sign of the fixed point $p$.
\end{enumerate}
\end{definition}

By definition, if a multigraph $\Gamma$ describes $M$, then $\Gamma$ is $n$-regular, where $\dim M=2n$. To associate a special multigraph, we need the following lemma.

\begin{lemma} \label{l33} Let the circle act on a $2n$-dimensional compact oriented manifold $M$ with a non-empty discrete fixed point set. Let 
\begin{center}
$W_+=\{w_{pi} \ | \ p \in M^{S^1}, \epsilon(p)=+1, 1 \leq i \leq n\}$
\end{center}
be the collection (multiset) of all the weights at the fixed points with sign $+1$ and 
\begin{center}
$W_-=\{w_{pi} \ | \ p \in M^{S^1}, \epsilon(p)=-1, 1 \leq i \leq n\}$
\end{center}
the collection of all the weights at the fixed points with sign $-1$. Rearrange elements in $W_+=\{a_1,a_2,\cdots,a_{i_1}\}$ and $W_-=\{b_1,b_2,\cdots,b_{i_2}\}$ so that $a_1 \leq a_2 \leq \cdots \leq a_{i_1}$ and $b_1 \leq b_2 \leq \cdots \leq b_{i_2}$. Then $a_1=b_1$ and $a_2=b_2$. Moreover, $a_i=a_2$ if and only if $b_i=a_2$. \end{lemma}

\begin{proof}
By Theorem \ref{t22}, the signature of $M$ is
\begin{center}
$\displaystyle \textrm{sign}(M) = \sum_{p \in M^{S^1}} \epsilon(p) \prod_{i=1}^{n} \frac{ 1+t^{w_{pi}}}{1-t^{w_{pi}}} = \sum_{p \in M^{S^1}} \epsilon(p) \prod_{i=1}^{n} [ (1+t^{w_{pi}}) ( \sum_{j=0}^{\infty} t^{j w_{pi}} )]$
$\displaystyle = \sum_{p \in M^{S^1}} \epsilon(p) \prod_{i=1}^{n} (  1+2 \sum_{j=1}^{\infty} t^{j w_{pi}} ).$
\end{center}
The signature of $M$ is independent of the indeterminate $t$ and is a constant. Therefore, comparing the term $2t^{a_1}$ that has the smallest positive exponent and positive coefficient and the term $-2t^{b_1}$ that has the smallest positive exponent and negative coefficient, we must have $a_1=b_1$. Cancel out all the terms $2t^{ja_1}$ and $-2t^{jb_1}$ for any $j>0$ in the equation above. Next, in what is left, we compare the terms whose exponent is equal to $a_2$. Let $i_0=\max\{i \ | \ a_i=a_2\}$. Then the fixed points $p$ with $\epsilon(p)=+1$ contribute the term $2(i_0-1)t^{a_2}$, since each $i$ contributes the term $2t^{a_2}$ for $2 \leq i \leq i_0$. Since this term $2(i_0-1)t^{a_2}$ must be cancelled out, it follows that $b_2=b_3=\cdots=b_{i_0}<b_{i_0+1}$. In particular, $a_2=b_2$. \end{proof}

\begin{rem}
In Lemma \ref{l33}, $a_3 \in W_+$ need not equal $b_3 \in W_-$, since the next smallest exponent in the equation above can be $a_1+a_2$. The simplest example is a linear action on $\mathbb{CP}^2$. Let $S^1$ act on $\mathbb{CP}^2$ by
\begin{center}
$g \cdot [z_0:z_1:z_2]=[z_0:g^a z_1:g^{a+b}z_2]$
\end{center}
for all $g \in S^1 \subset \mathbb{C}$, for some positive integers $a$ and $b$. The action has 3 fixed points $[1:0:0]$, $[0:1:0]$, and $[0:0:1]$, which have the weights $\{a,a+b\}$, $\{-a,b\}$, and $\{-b,-a-b,\}$ as complex $S^1$-representations, and hence fixed point data $\{+,a,a+b\}$, $\{-,a,b\}$, and $\{+,b,a+b\}$, respectively. In the notations as in Lemma \ref{l33}, $W_+=\{a,b,a+b,a+b\}$ and $W_-=\{a,b\}$.
\end{rem}

For a multigraph, by a self-loop we mean an edge whose vertices coincide.

\begin{lemma} \label{l35} Let the circle act on a compact oriented manifold $M$ with a discrete fixed point set. Then there exists a signed labeled multigraph $\Gamma$ describing $M$ that satisfies the following properties:
\begin{enumerate}
\item If the label of an edge is equal to $a_1$ or $a_2$ where $a_1$ and $a_2$ are as in Lemma \ref{l33}, its vertices have different signs. In particular, $\Gamma$ has at least two edges whose vertices have different signs.
\item If the label $w$ of an edge is strictly bigger than $a_2$, the vertices of the edge lie in the same connected component of $M^{\mathbb{Z}_w}$.
\item The multigraph $\Gamma$ has no self-loops.
\end{enumerate}
\end{lemma}

\begin{proof}
First, assign a vertex to each fixed point $p \in M^{S^1}$. To every vertex, assign the sign of the corresponding fixed point.

Let $W_+=\{a_1,a_2,\cdots,a_{i_1}\}$ and $W_-=\{b_1,b_2,\cdots,b_{i_2}\}$ be as in Lemma \ref{l33}. Suppose that a fixed point $p_1$ with $\epsilon(p_1)=1$ has weight $a_1$. By Lemma \ref{l33}, $b_1=a_1$, i.e., there exists a fixed point $q_1$ with $\epsilon(q_1)=-1$ that has weight $a_1$. We draw an edge between $p_1$ and $q_1$, giving label $a_1$. Next, let $i_0=\max\{i \, | \, a_i=a_2\}$. By Lemma \ref{l33}, $b_2=b_3=\cdots=b_{i_0}<b_{i_0+1}$. Therefore, for each $j$ such that $2 \leq j \leq i_0$, if a fixed point $p$ with $\epsilon(p)=+1$ has weight $a_j (=a_2)$ and a fixed point $q$ with $\epsilon(q)=-1$ has weight $b_j (=a_j)$, then draw an edge between $p$ and $q$ giving label $a_j$, which is equal to $a_2$. For each fixed point, one weight is used to draw only one edge. This proves Claim (1).

Let $w > a_2$ be an integer. Suppose that a fixed point $p_0$ has weight $w$. Consider a connected component $F$ of $M^{\mathbb{Z}_w}$ that contains $p_0$, which is orientable by Lemma \ref{l25}. Fix an orientation of $F$. The circle action on $M$ restricts to act on $F$, and the fixed point set of this $S^1$-action on $F$ is equal to $F^{S^1}=F \cap M^{S^1}$, which is discrete. If $q \in F^C$, every weight in $T_qF$ is a multiple of $w$. As for the circle action on $M$, for the $S^1$-action on $F$, let 
\begin{center}
$W_{+,F}=\{w_{pi} \ | \ p \in F^{S^1}, \epsilon_F(p)=+1, w_{pi}\textrm{ is a weight in }T_pF\}=\{c_1,c_2,\cdots,c_{i_3}\}$,

$W_{-,F}=\{w_{pi} \ | \ p \in F^{S^1}, \epsilon_F(p)=-1, w_{pi} \textrm{ is a weight in } T_pF\}=\{d_1,d_2,\cdots,d_{i_4}\}$,
\end{center}
be multisets, where $c_1 \leq c_2 \leq \cdots \leq c_{i_3}$ and $d_1 \leq d_2 \leq \cdots \leq d_{i_4}$. Let $i_5=\max\{j \, | \, c_j=c_1\}$. Applying Lemma \ref{l33} to the $S^1$-action on $F$, we have $d_1=\cdots=d_{i_5}<d_{i_5+1}$. Therefore, for each $1 \leq j \leq i_5$, if a fixed point $p \in F^{S^1}$ with $\epsilon_F(p)=+1$ has weight $c_j$ and a fixed point $q \in F^{S^1}$ with $\epsilon_F(q)=-1$ has weight $d_j (=c_j)$, then draw an edge between $p$ and $q$ giving label $c_j$.

Repeat the above for every positive integer $w$ such that $w>a_2$ and for every connected component of $M^{\mathbb{Z}_w}$. This proves Claim (2). The resulting multigraph describes $M$ and the whole procedure does not cause any self-loops. This proves Claim (3). \end{proof}

For a compact oriented $S^1$-manifold with a discrete fixed point set, let $\Gamma_1$ and $\Gamma_2$ be two signed labeled multigraphs describing $M$, where $\Gamma_1$ is arbitrary but $\Gamma_2$ satisfies the properties in Lemma \ref{l35}. Since both of them describe $M$, they encode the same fixed point data; the same sign and the same multiset of the weights at every fixed point $p \in M^{S^1}$. On the other hand, only considering multigraphs that satisfy the properties in Lemma \ref{l35} reduces the proof of Theorem \ref{t11} as we shall we in the next section.

\section{6 dimension and 4 fixed points: possible multigraphs} \label{s4}

In this section, we show what kind of signed labeled multigraph can occur as a multigraph describing a 6-dimensional compact oriented $S^1$-manifold having 4 fixed points.

\begin{figure}
\begin{subfigure}[b][4.5cm][s]{.27\textwidth}
\centering
\vfill
\begin{tikzpicture}[state/.style ={circle, draw}]
\node[state] (a) {$p_1,+$};
\node[state] (b) [right=of a] {$p_2,+$};
\node[state] (c) [above=of a] {$p_3,-$};
\node[state] (d) [right=of c] {$p_4,-$};
\path (a)  edge node[left] {$b$} (c);
\path (a)  [bend left =30]edge node[left] {$a$} (c);
\path (c)  [bend left =30]edge node[right] {$c$} (a);
\path (b)  edge node[left] {$e$} (d);
\path (b)  [bend left =30]edge node[left] {$d$} (d);
\path (d)  [bend left =30]edge node[right] {$f$} (b);
\end{tikzpicture}
\vfill
\caption{Case A}
\label{fig1-1}
\vspace{\baselineskip}
\end{subfigure}\qquad
\begin{subfigure}[b][4.5cm][s]{.27\textwidth}
\centering
\vfill
\begin{tikzpicture}[state/.style ={circle, draw}]
\node[state] (a) {$p_1,+$};
\node[state] (b) [right=of a] {$p_2,+$};
\node[state] (c) [above=of a] {$p_3,-$};
\node[state] (d) [right=of c] {$p_4,-$};
\path (a)  [bend right =20]edge node[left] {$c$} (c);
\path (b)  edge node[below] {$a$} (a);
\path (a)  [bend left =20]edge node[left] {$b$} (c);
\path (d)  [bend left =20]edge node[left] {$e$} (b);
\path (b)  [bend left =20]edge node[left] {$d$} (d);
\path (c)  edge node[above] {$f$} (d);
\end{tikzpicture}
\vfill
\caption{Case B}
\label{fig1-2}
\vspace{\baselineskip}
\end{subfigure}\qquad
\begin{subfigure}[b][4.5cm][s]{.27\textwidth}
\centering
\vfill
\begin{tikzpicture}[state/.style ={circle, draw}]
\node[state] (a) {$p_1,+$};
\node[state] (b) [right=of a] {$p_2,+$};
\node[state] (c) [above=of a] {$p_3,-$};
\node[state] (d) [right=of c] {$p_4,-$};
\path (a)  edge node[left] {$b$} (c);
\path (b)  edge node[pos=.2, right, sloped, rotate=90] {$d$} (c);
\path (c)  [bend right =30]edge node[left] {$a$} (a);
\path (d)  edge node[right] {$e$} (b);
\path (b)  [bend right =30]edge node[right] {$f$} (d);
\path (a)  edge node[pos=.2, right, sloped, rotate=270] {$c$} (d);
\end{tikzpicture}
\vfill
\caption{Case C}
\label{fig1-3}
\vspace{\baselineskip}
\end{subfigure}\qquad
\begin{subfigure}[b][4.5cm][s]{.27\textwidth}
\centering
\vfill
\begin{tikzpicture}[state/.style ={circle, draw}]
\node[state] (a) {$p_1,+$};
\node[state] (b) [right=of a] {$p_2,+$};
\node[state] (c) [above=of a] {$p_3,-$};
\node[state] (d) [right=of c] {$p_4,-$};
\path (a)  [bend left =10]edge node[above] {$b$} (b);
\path (a)  edge node [left] {$c$} (c);
\path (b)  [bend left =10] edge node [below] {$a$} (a);
\path (b)  edge node [right] {$d$} (d);
\path (d)  [bend left =10]edge node [below] {$e$} (c);
\path (c)  [bend left =10]edge node [above] {$f$} (d);
\end{tikzpicture}
\vfill
\caption{Case D}
\label{fig1-4}
\vspace{\baselineskip}
\end{subfigure}\qquad
\begin{subfigure}[b][4.5cm][s]{.27\textwidth}
\centering
\vfill
\begin{tikzpicture}[state/.style ={circle, draw}]
\node[state] (a) {$p_1,+$};
\node[state] (b) [right=of a] {$p_2,+$};
\node[state] (c) [above=of a] {$p_3,-$};
\node[state] (d) [right=of c] {$p_4,-$};
\path (b)  edge node[below] {$a$} (a);
\path (a)  edge node[pos=.2, left, sloped, rotate=270] {$c$} (d);
\path (a) edge node [left] {$b$} (c);
\path (d)  edge node [right] {$e$} (b);
\path (b) edge node[pos=.2, right, sloped, rotate=90] {$d$} (c);
\path (c) edge node [above] {$f$} (d);
\end{tikzpicture}
\vfill
\caption{Case E}
\label{fig1-5}
\vspace{\baselineskip}
\end{subfigure}\qquad
\caption{Possible (special) multigraphs}\label{fig1}
\end{figure}
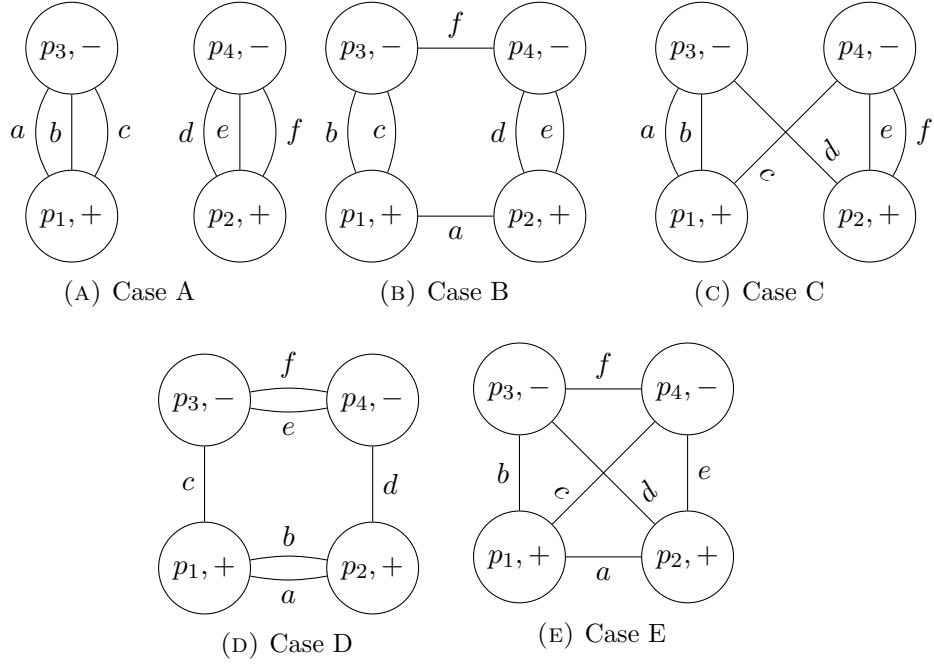

\begin{lemma} \label{l41} Let the circle act on a 6-dimensional compact oriented manifold $M$ with 4 fixed points. Then one of the figures in Figure \ref{fig1} occurs as a signed labeled multigraph describing $M$ that satisfies the properties in Lemma \ref{l35}. In particular, one of the following occurs as the fixed point data of $M$ for some positive integers $a$, $b$, $c$, $d$, $e$, and $f$.
\begin{enumerate}[(A)]
\item $\{+,a,b,c\}, \{+,d,e,f\}, \{-,a,b,c\}, \{-,d,e,f\}$ (Figure \ref{fig1-1}).
\item $\{+,a,b,c\}, \{+,a,d,e\}, \{-,b,c,f\}, \{-,d,e,f\}$ (Figure \ref{fig1-2}).
\item $\{+,a,b,c\}, \{+,d,e,f\}, \{-,a,b,d\}, \{-,c,e,f\}$ (Figure \ref{fig1-3}).
\item $\{+,a,b,c\}, \{+,a,b,d\}, \{-,c,e,f\}, \{-,d,e,f\}$ (Figure \ref{fig1-4}).
\item $\{+,a,b,c\}, \{+,a,d,e\}, \{-,b,d,f\}, \{-,c,e,f\}$ (Figure \ref{fig1-5}).
\end{enumerate}
\end{lemma}

\begin{proof}
Since $\dim M \neq 0 \mod 4$, the L-genus of $M$ is equal to 0, which is equal to the signature of $M$. Taking $t=0$ in Theorem \ref{t22}, 
\begin{center}
$\displaystyle 0=\textrm{sign}(M)=\sum_{p \in M^{S^1}} \epsilon(p)$.
\end{center}
This implies that there are two fixed points $p_1$ and $p_2$ with sign $+1$ and two fixed points $p_3$ and $p_4$ with sign $-1$. By Lemma \ref{l35}, there exists a signed labeled multigraph $\Gamma$ describing $M$ that satisfies the following properties:
\begin{enumerate}
\item If the label of an edge is equal to $a_1$ or $a_2$ where $a_1$ and $a_2$ are as in Lemma \ref{l33}, its vertices have different signs. In particular, $\Gamma$ has at least two edges whose vertices have different signs.
\item If the label $w$ of an edge is strictly bigger than $a_2$, the vertices of the edge lie in the same connected component of $M^{\mathbb{Z}_w}$.
\item The multigraph $\Gamma$ has no self-loops.
\end{enumerate}

For $i \in \{2,3,4\}$, let $m_i$ be the number of edges between $p_1$ and $p_i$. By permuting $p_1$ and $p_2$, and by permuting $p_3$ and $p_4$ if necessary, we may assume that $m_3 \geq 1$ (by Property (1) above) and $m_3 \geq m_4$ (by permuting $p_3$ and $p_4$ if necessary). Since $m_i \geq 0$ for any $i$ and $m_2+m_3+m_4=3$, we have the following cases for $(m_2,m_3,m_4)$; $(0,3,0)$, $(1,2,0)$, $(0,2,1)$, $(2,1,0)$, and $(1,1,1)$ (in the order that $m_3$ is non-increasing). Each case in order corresponds to a multigraph in Figure \ref{fig1}, when we complete the multigraph that satisfies the properties in Lemma \ref{l35}. \end{proof}

\section{Proof of the main theorem} \label{s5}

In this section, we prove our main result, Theorem \ref{t11}, by determining the fixed point data for each multigraph in Figure \ref{fig1}. Since the fixed point data of Figure \ref{fig1-1} is precisely Case (1) of Theorem \ref{t11}, there is nothing to prove. Figures \ref{fig1-2}, \ref{fig1-3}, and \ref{fig1-4} are easy to deal with; the fixed point data in any of these figures falls into Case (1) of Theorem \ref{t11}. Figure \ref{fig1-5} (Case E) is the most difficult case and requires complicated arguments. The fixed point data from Figure \ref{fig1-5} either belongs to Case (1) or is precisely Case (2) of Theorem \ref{t11}, depending on the relationship between the weights at the fixed points. For Figure \ref{fig1-5}, when the fixed point data is the same as in Case (2) of Theorem \ref{t11}, the multigraph \ref{fig1-5} describes a standard linear action on $\mathbb{CP}^3$.

\begin{lemma} \label{l51} Let the circle act on a 6-dimensional compact oriented manifold $M$ with 4 fixed points. Suppose that Figure \ref{fig1-2} describes $M$ and satisfies the properties in Lemma \ref{l35}. Then $a=f$; the fixed point data of $M$ belongs to Case (1) of Theorem \ref{t11}.  \end{lemma}

\begin{proof} The fixed point data of $M$ is
\begin{center}
$\{+,a,b,c\}$, $\{+,a,d,e\}$, $\{-,b,c,f\}$, $\{-,d,e,f\}$.
\end{center}
For a dimensional reason, the push-forward of the equivariant cohomology class 1 in equivariant cohomology vanishes; $\int_M 1=0$ ($\int_M:H_{S^1}^i(M) \to H^{i- \dim M}(\mathbb{CP}^\infty)$ for all $i$). Therefore, by Theorem \ref{t21},
\begin{center}
$\displaystyle 0=\int_M 1=\sum_{p \in M^{S^1}} \epsilon(p) \frac{1}{\prod_{i=1}^n w_{pi}}=\frac{1}{abc}+\frac{1}{ade}-\frac{1}{bcf}-\frac{1}{def}=\left(\frac{1}{bc}+\frac{1}{de}\right)\left(\frac{1}{a}-\frac{1}{f}\right).$
\end{center}
Thus $a=f$ and hence the fixed point data of $M$ belongs to Case (1) of Theorem \ref{t11}. \end{proof}

\begin{lemma} \label{l52} Let the circle act on a 6-dimensional compact oriented manifold $M$ with 4 fixed points. Suppose that Figure \ref{fig1-3} describes $M$ and satisfies the properties in Lemma \ref{l35}. Then $c=d$ or $\{a,b\}=\{e,f\}$; the fixed point data of $M$ belongs to Case (1) of Theorem \ref{t11}. 
\end{lemma}

\begin{proof} The fixed point data of $M$ is
\begin{center}
$\{+,a,b,c\}, \{+,d,e,f\}, \{-,a,b,d\}, \{-,c,e,f\}$. 
\end{center}
Since $\dim M \neq 0 \mod 4$, the signature of $M$ is equal to 0. By Theorem \ref{t22},
\begin{center}
$\displaystyle 0=\textrm{sign}(M) = \sum_{p \in M^{S^1}} \epsilon(p) \prod_{i=1}^{3} \frac{ 1+t^{w_{pi}}}{1-t^{w_{pi}}}= \sum_{p \in M^{S^1}} \epsilon(p) \prod_{i=1}^{3} [ (1+t^{w_{pi}}) ( \sum_{j=0}^{\infty} t^{j w_{pi}} )]$
$\displaystyle = \sum_{p \in M^{S^1}} \epsilon(p) \prod_{i=1}^{3} (1+2\sum_{j=1}^{\infty}t^{j w_{pi}} )$

$\displaystyle = (1+2 \sum_{j=1}^{\infty} t^{ja})(1+2 \sum_{j=1}^{\infty} t^{j b} ) (1+2 \sum_{j=1}^{\infty} t^{j c})+(1+2 \sum_{j=1}^{\infty}t^{jd})(1+2 \sum_{j=1}^{\infty} t^{je})(1+2\sum_{j=1}^{\infty}t^{j f})-(1+2 \sum_{j=1}^{\infty} t^{ja})(1+2 \sum_{j=1}^{\infty} t^{j b} )(1+2 \sum_{j=1}^{\infty} t^{jd})-(1+2 \sum_{j=1}^{\infty} t^{jc})( 1+2\sum_{j=1}^{\infty}t^{j e})(1+2 \sum_{j=1}^{\infty} t^{j f})$

$=\displaystyle [(1+2 \sum_{j=1}^{\infty} t^{j c})-(1+2 \sum_{j=1}^{\infty} t^{jd})][(1+2 \sum_{j=1}^{\infty} t^{ja})(1+2 \sum_{j=1}^{\infty} t^{j b} )-(1+2 \sum_{j=1}^{\infty} t^{j e} )(  1+2\sum_{j=1}^{\infty}t^{j f})]$.
\end{center}
It follows that $c=d$ or $\{a,b\}=\{e,f\}$. In either case, the fixed point data of $M$ belongs to Case (1) of Theorem \ref{t11}. \end{proof}

\begin{lemma} \label{l53} Let the circle act on a 6-dimensional compact oriented manifold $M$ with 4 fixed points. Suppose that Figure \ref{fig1-4} describes $M$ and satisfies the properties in Lemma \ref{l35}. Then $\{a,b\}=\{e,f\}$; the fixed point data of $M$ belongs to Case (1) of Theorem \ref{t11}.  \end{lemma}

\begin{proof} The proof is similar to Lemma \ref{l52}. The fixed point data of $M$ is
\begin{center}
$\{+,a,b,c\}, \{+,a,b,d\}, \{-,c,e,f\}, \{-,d,e,f\}$.
\end{center}
By Theorem \ref{t22}, we deduce that
\begin{center}
$\displaystyle 0=\textrm{sign}(M)$

$=\displaystyle [(1+2 \sum_{j=1}^{\infty} t^{j c})+(1+2 \sum_{j=1}^{\infty} t^{jd})][(1+2 \sum_{j=1}^{\infty} t^{ja})(1+2 \sum_{j=1}^{\infty} t^{j b} )-(1+2 \sum_{j=1}^{\infty} t^{j e} )(  1+2\sum_{j=1}^{\infty}t^{j f})]$.
\end{center}
It follows that $\{a,b\}=\{e,f\}$. The fixed point data of $M$ then belongs to Case (1) of Theorem \ref{t11}. \end{proof}

We are left with Figure \ref{fig1-5} in Lemma \ref{l41}. However, Figure \ref{fig1-5} requires many subcases to consider and more techniques.

\begin{lemma} \label{l54} Let the circle act on a 6-dimensional compact oriented manifold $M$ with 4 fixed points. Suppose that Figure \ref{fig1-5} describes $M$ and satisfies the properties in Lemma \ref{l35}. Assume that $b$ is the largest weight and $\dim F=4$, where $F$ is a connected component of $M^{\mathbb{Z}_b}$ that contains $p_1$ and $p_3$. Then the fixed point data of $M$ belongs to Case (1) of Theorem \ref{t11}. \end{lemma}

\begin{proof}
Suppose that $b$ is the largest weight. Since the multigraph Figure \ref{fig1-5} has an edge whose vertices have the same sign, Properties 1 and 2 of Lemma \ref{l35} imply that the largest weight is strictly bigger than $a_2 \in W_+$ where $a_2$ and $W_+$ are as in Lemma \ref{l33}; in particular $b$ is strictly bigger than $a_2$.

Since $\dim F=4$, exactly 2 weights at $p_1$ are divisible by (and hence equal to) $b$. That is, either $a=b$ or $c=b$. Suppose that $a=b$. Since $p_1$ and $p_2$ are connected by the edge with label $a$ ($=b$), by Property 2 of Lemma \ref{l35}, $p_2$ also lies in $F$. The circle action on $M$ restricts to a circle action on $F$. By Lemma \ref{l25}, $F$ is orientable; fix an orientation of $F$. If $p \in F^{S^1}$, every weight in $T_pF$ is a multiple of (and hence equal to) $b$. Applying Lemma \ref{l28} to the $S^1$-action on $F$, the $S^1$-action on $F$ has an even number of fixed points. Since $p_1,p_2,p_3 \in F^{S^1}$ and $F^{S^1} \subset M^{S^1}$, this implies that $p_4 \in F^{S^1}$. This means that each $p_i$ has exactly two weights that are equal to $b$, and hence it follows that $e=b$ and $f=b$. Then the fixed point data of $M$ is
\begin{center}
$\{+,b,b,c\}$, $\{+,b,b,d\}$, $\{-,b,b,d\}$, $\{-,b,b,c\}$.
\end{center}
This belongs to Case (1) of Theorem \ref{t11}.

Next, suppose that $c=b$. By an analogous argument as above, $c=b$ implies $p_4 \in F$ by Property 2 of Lemma \ref{l35}, and applying Lemma \ref{l28} to the restriction of the circle action on $M$ to $F$, it follows that $p_i \in F$ for all $i$, and hence $d=e=b$. Then the fixed point data of $M$ is
\begin{center}
$\{+,a,b,b\}$, $\{+,a,b,b\}$, $\{-,b,b,f\}$, and $\{-,b,b,f\}$.
\end{center}
By Theorem \ref{t22}, we deduce that
\begin{center}
$\displaystyle 0=\textrm{sign}(M)$

$=\displaystyle 2[(1+2 \sum_{j=1}^{\infty} t^{j a})-(1+2 \sum_{j=1}^{\infty} t^{jf})](1+2 \sum_{j=1}^{\infty} t^{jb})^2$.
\end{center}
It follows that $a=f$. The fixed point data of $M$ belongs to Case (1) of Theorem \ref{t11}. \end{proof}

\begin{lemma} \label{l55} Let the circle act on a 6-dimensional compact oriented manifold $M$ with 4 fixed points. Suppose that Figure \ref{fig1-5} describes $M$ and satisfies the properties in Lemma \ref{l35}. Assume that $a$ is the largest weight and $\dim F=4$, where $F$ is a connected component of $M^{\mathbb{Z}_a}$ that contains $p_1$ and $p_2$. Then the fixed point data of $M$ belongs to Case (1) of Theorem \ref{t11}. \end{lemma}

\begin{proof} The proof is analogous to that of Lemma \ref{l54}. Since the multigraph Figure \ref{fig1-5} has an edge whose vertices have the same sign, Properties 1 and 2 of Lemma \ref{l35} imply that the largest weight is strictly bigger than $a_2 \in W_+$; in particular $a$ is strictly bigger than $a_2$. Since $\dim F=4$, exactly 2 weights at $p_1$ must be divisible by (and hence equal to) $a$. That is, either $b=a$ or $c=a$. If $b=a$, this case is the same as the case in Lemma \ref{l54} where $b$ is the largest weight and $a=b$. As in the proof of Lemma \ref{l54}, the fixed point data of $M$ in this case belongs to Case (1) of Theorem \ref{t11}.

Next, assume $c=a$. Since $p_1$ and $p_4$ are connected by the edge with label $c$ ($=a$), by Property 2 of Lemma \ref{l35}, $p_4$ also lies in $F$. The circle action on $M$ restricts to a circle action on $F$. By Lemma \ref{l25}, $F$ is orientable; fix an orientation of $F$. If $p \in F^{S^1}$, every weight in $T_pF$ is a multiple of (and hence equal to) $a$. Applying Lemma \ref{l28} to the $S^1$-action on $F$, the $S^1$-action on $F$ has an even number of fixed points. Since $p_1,p_2,p_4 \in F^{S^1}$ and $F^{S^1} \subset M^{S^1}$, this implies that $p_3 \in F^{S^1}$. This means that each $p_i$ has exactly two weights that are equal to $a$, and hence it follows that $d=f=a$. The fixed point data of $M$ is then $\{+,a,a,b\}$, $\{+,a,a,e\}$, $\{-,a,a,b\}$, $\{-,a,a,e\}$. This belongs to Case (1) of Theorem \ref{t11}. \end{proof}

\begin{lemma} \label{l56} Let the circle act on a 6-dimensional compact oriented manifold $M$ with 4 fixed points. Suppose that Figure \ref{fig1-5} describes $M$ and satisfies the properties in Lemma \ref{l35}. Assume that $b$ is the largest weight and $\dim F=2$, where $F$ is a connected component of $M^{\mathbb{Z}_b}$ that contains $p_1$ and $p_3$. Then one of the following holds for the fixed point data of $M$:
\begin{enumerate}
\item $\{+,A,B,C\},\{-,A,B,C\},\{+,D,E,F\},\{-,D,E,F\}$ for some positive integers $A$, $B$, $C$, $D$, $E$, and $F$.
\item $\{+,A,A+B,A+B+C\}, \{-,A,B,B+C\}, \{+,B,C,A+B\}, \{-,C,B+C,A+B+C\}$. for some positive integers $A$, $B$, and $C$. 
\end{enumerate}
\end{lemma}

\begin{proof} Since there is an edge whose vertices have the same sign, Properties 1 and 2 of Lemma \ref{l35} imply that the largest weight is strictly bigger than $a_2 \in W_+$; in particular $b$ is strictly bigger than 1. Since $\dim F=2$, the other weights $a$ and $c$ at $p_1$ are smaller than $b$. Similarly, the other weights $d$ and $f$ at $p_3$ are smaller than $b$. We apply Lemma \ref{l26} for weight $b$ and the two fixed poitns $p_1$ and $p_3$ that lie in the same connected component $F$ of $M^{\mathbb{Z}_b}$; there exist a bijection $\sigma:\{1,2\} \to \{1,2\}$ and  a function $\nu:\{1,2\} \to \{-1,1\}$ such that 
\begin{center}
$w_{p_1,i} \equiv \nu(i) \cdot w_{p_3,\sigma(i)} \mod b$ for $1 \leq i \leq 2$, 
\end{center}
where $\{w_{p_1,1},w_{p_1,2}\}=\{a,c\}$ and $\{w_{p_3,1},w_{p_3,2}\}=\{d,f\}$. Moreover, $\epsilon(p_1)=\epsilon(p_3) \cdot (-1)^{\nu_{p_1,p_3}^- +1}$, where $\nu_{p_1,p_3}^-$ is the number of $i$'s with $\nu(i)=-1$. Since $\epsilon(p_1)=-\epsilon(p_3)$, it follows that $\nu_{p_1,p_3}^-=0$ or $2$. These imply that exactly one of the following holds:
\begin{enumerate}[(i)]
\item $a \equiv d \mod b$ and $c \equiv f \mod b$. ($\sigma(1)=1$, $\sigma(2)=2$, $\nu_{p_1,p_3}^-=0$)
\item $a \equiv f \mod b$ and $c \equiv d \mod b$. ($\sigma(1)=2$, $\sigma(2)=1$, $\nu_{p_1,p_3}^-=0$)
\item $a \equiv -d \mod b$ and $c \equiv -f \mod b$. ($\sigma(1)=1$, $\sigma(2)=2$, $\nu_{p_1,p_3}^-=2$)
\item $a \equiv -f \mod b$ and $c \equiv -d \mod b$. ($\sigma(1)=2$, $\sigma(2)=1$, $\nu_{p_1,p_3}^-=2$)
\end{enumerate}

Assume that Case (i) holds. Since $a,c,d,f<b$, this implies that $a=d$ and $c=f$. The fixed point data of $M$ is then
\begin{center}
$\{+,a,b,c\}$, $\{+,a,a,e\}$, $\{-,b,a,c\}$, $\{-,c,e,c\}$.
\end{center}
By Theorem \ref{t21},
\begin{center}
$\displaystyle 0=\int_M 1=\sum_{p \in M^{S^1}} \frac{1}{\prod_{i=1}^3 w_{pi}}=\frac{1}{abc}+\frac{1}{a^2e}-\frac{1}{abc}-\frac{1}{c^2e}$.
\end{center}
This implies that $a=c$. The fixed point data of $M$ belongs to Case (1) of the lemma.

Assume that Case (ii) holds. Then $a=f$ and $c=d$. The fixed point data of $M$ is
\begin{center}
$\{+,a,b,c\}$, $\{+,a,c,e\}$, $\{-,b,c,a\}$, $\{-,c,e,a\}$.
\end{center}
This belongs to Case (1) of the lemma.

Assume that Case (iii) holds. Then $a+d=b$ and $c+f=b$, i.e., $d=b-a$ and $f=b-c$. With these, by Theorem \ref{t21},
\begin{center}
$\displaystyle 0=\int_M 1=\sum_{p \in M^{S^1}} \frac{1}{\prod_{i=1}^3 w_{pi}}=\frac{1}{abc}+\frac{1}{a(b-a)e}-\frac{1}{b(b-a)(b-c)}-\frac{1}{ce(b-c)}$.
\end{center}
Clearing the denominators and simplifying, we conclude that $e=a-c$. Then the fixed point data of $M$ is
\begin{center}
$\{+,a,b,c\}$, $\{+,a,b-a,a-c\}$, $\{-,b,b-a,b-c\}$, $\{-,c,a-c,b-c\}$.
\end{center}
Let $A=c$, $B=a-c$, $C=b-a$. We can rewrite the fixed point data of $M$;
\begin{center}
$\{+,A+B,A+B+C,A\}$, $\{+,A+B,C,B\}$, $\{-,A+B+C,C,B+C\}$, $\{-,A,B,B+C\}$.
\end{center}
This is precisely Case (2) of the lemma.

Assume that Case (iv) holds. Then $a+f=b$ and $c+d=b$. Let $a_1$ and $a_2$ be as in Lemma \ref{l33}. By Property 2 of Lemma \ref{l35}, we have $a_2<a$ and $a_2<f$, since the edges of the labels $a$ and $f$ have vertices with the same sign. Next, because $b$, $c$, $d$, and $e$ are the only labels whose edges have vertices with different signs and because $b$ is the largest weight, by Property 1 of Lemma \ref{l35}, we must have $\{a_1,a_2\} \subset \{c,d,e\}$. We cannot have $\{a_1,a_2\}=\{c,d\}$, since then we have $b=c+d=a_1+a_2<a+f=b$. It follows that either $\{c,e\}=\{a_1,a_2\}$ or $\{d,e\}=\{a_1,a_2\}$. By reversing the orientation of $M$, which amounts to changing the sign of each fixed point (in other words, by the vertical symmetry of the multigraph Figure \ref{fig1-5}), we may assume without loss of generality that $\{c,e\}=\{a_1,a_2\}$. Since $\{c,e\}=\{a_1,a_2\}$, $a_2<a$, $a_2<f$, $a+f=b$, and $c+d=b$, we have 
\begin{equation}
\max\{c,e\} < \min\{a,f\} \leq \frac{b}{2} \leq \max\{a,f\}<d<b.
\end{equation}

Since $d>a_2$ and Figure \ref{fig1-5} satisfies the properties in Lemma \ref{l35}, $p_2$ and $p_3$ lie in the same connected component $F_d$ of $M^{\mathbb{Z}_d}$. As above, we apply Lemma \ref{l26} for the weight $d$; because no multiples of $d$ occur as weights other than $d$ itself and $\epsilon(p_2)=-\epsilon(p_3)$, it follows that $\nu_{p_2,p_3}^-=0$ or $2$, and exactly one of the following holds for the fixed point data of $p_2$ and that of $p_3$.
\begin{enumerate}
\item $a \equiv b \mod d$ and $e \equiv f \mod d$. ($\sigma(1)=1$, $\sigma(2)=2$, $\nu_{p_2,p_3}^-=0$)
\item $a \equiv f \mod d$ and $e \equiv b \mod d$. ($\sigma(1)=2$, $\sigma(2)=1$, $\nu_{p_2,p_3}^-=0$)
\item $a \equiv -b \mod d$ and $e \equiv -f \mod d$. ($\sigma(1)=1$, $\sigma(2)=2$, $\nu_{p_2,p_3}^-=2$)
\item $a \equiv -f \mod d$ and $e \equiv -b \mod d$. ($\sigma(1)=2$, $\sigma(2)=1$, $\nu_{p_2,p_3}^-=2$)
\end{enumerate}

Suppose that $a \equiv b \mod d$ and $e \equiv f \mod d$. Since $a<b=c+b-c<a+b-c=a+d$, we cannot have $a \equiv b \mod d$.

Suppose that $a \equiv f \mod d$ and $e \equiv b \mod d$. It follows that $a=f$ and $e+d=b$. Since $f=b-a$ and $d=b-c$, this means that $b=2a$ and $c=e$. The fixed point data of $M$ is then
\begin{center}
$\{+,a,2a,c\}$, $\{+,a,2a-c,c\}$, $\{-,2a,2a-c,a\}$, $\{-,c,c,a\}$.
\end{center}
By Theorem \ref{t22}, we deduce that
\begin{center}
$\displaystyle 0=\textrm{sign}(M)$

$=\displaystyle (1+2 \sum_{j=1}^{\infty} t^{ja})(1+2 \sum_{j=1}^{\infty} t^{jc})[(1+2 \sum_{j=1}^{\infty} t^{j (2a)})+(1+2 \sum_{j=1}^{\infty} t^{j(2a-c)})]-(1+2 \sum_{j=1}^{\infty} t^{ja})[(1+2 \sum_{j=1}^{\infty} t^{j(2a)})(1+2 \sum_{j=1}^{\infty} t^{j(2a-c)})+(1+2 \sum_{j=1}^{\infty} t^{jc})^2]$.
\end{center}
Dividing the equation by $(1+2 \sum_{j=1}^{\infty} t^{ja})$ and simplifying, we have
\begin{center}
$\displaystyle (1+2 \sum_{j=1}^{\infty} t^{jc})[(1+2 \sum_{j=1}^{\infty} t^{j (2a)})+(1+2 \sum_{j=1}^{\infty} t^{j(2a-c)})]=(1+2 \sum_{j=1}^{\infty} t^{j(2a)})(1+2 \sum_{j=1}^{\infty} t^{j(2a-c)})+(1+2 \sum_{j=1}^{\infty} t^{jc})^2$.
\end{center}
The equation only holds if $a=c$, but $c<a$ and hence this leads to a contradiction.

Suppose that $a \equiv -b \mod d$ and $e \equiv -f \mod d$. Since $e,f<d$, $e \equiv -f \mod d$ implies that $e+f=d$. Since $f=b-a$, this means that $e+b=a+d$. Next, $a \equiv -b \mod d$ implies that $a+b=2d$ because $a<d<b<2d$. Hence, $b=2d-a$, $c=b-d=d-a$, $e=a+d-b=2a-d$, and $f=b-a=2d-2a$. The fixed point data of $M$ is then
\begin{center}
$\{+,a,2d-a,d-a\}$, $\{+,a,d,2a-d\}$, $\{-,2d-a,d,2d-2a\}$, $\{-,d-a,2a-d,2d-2a\}$.
\end{center}
By Theorem \ref{t21},
\begin{center}
$\displaystyle 0=\int_M 1=\sum_{p \in M^{S^1}} \frac{1}{\prod_{i=1}^3 w_{pi}}=\frac{1}{a(2a-d)(d-a)}+\frac{1}{ad(2a-d)}-\frac{1}{(2d-a)d(2d-2a)}-\frac{1}{(d-a)(2a-d)(2d-2a)}$.
\end{center}
We simplify this to $(2d-a)(d-2a)=0$, but this cannot hold because $2d-a=b>0$ and $2a-d=e>0$.

Suppose that $a \equiv -f \mod d$ and $e \equiv -b \mod d$. In this case, $a \equiv -f \mod d$ cannot hold, since $-f=a-b$ and $\frac{b}{2}<d<b$. \end{proof}

\begin{lemma} \label{l57} Let the circle act on a 6-dimensional compact oriented manifold $M$ with 4 fixed points. Suppose that Figure \ref{fig1-5} describes $M$ and satisfies the properties in Lemma \ref{l35}. Assume that $a$ is the largest weight and $\dim F=2$, where $F$ is a connected component of $M^{\mathbb{Z}_a}$ that contains $p_1$ and $p_2$. Then the fixed point data of $M$ belongs to Case (1) of Theorem \ref{t11}. \end{lemma}

\begin{proof}
The proof is similar to that of Lemma \ref{l56}. Suppose that $a$ is the largest weight. By Property 2 of Lemma \ref{l35}, since its edge has vertices of the same sign, the largest weight is strictly bigger than $a_2 \in W_+$ where $a_2$ and $W_+$ are as in Lemma \ref{l33}; in particular $a$ is strictly bigger than 1.

Since $\dim F=2$, the other weights $b$ and $c$ at $p_1$ are smaller than $a$, and the other weights $d$ and $e$ at $p_2$ are smaller than $a$. Applying Lemma \ref{l26} for weight $a$ and the two fixed points $p_1$ and $p_2$ that lie in the same connected component $F$ of $M^{\mathbb{Z}_a}$, there exist a bijection $\sigma:\{1,2\} \to \{1,2\}$ and a map $\nu:\{1,2\} \to \{-1,1\}$ such that 
\begin{center}
$w_{p_1,i} \equiv \nu(i) \cdot w_{p_2,\sigma(i)} \mod b$ for $1 \leq i \leq 2$, 
\end{center}
where $\{w_{p_1,1},w_{p_1,2}\}=\{b,c\}$ and $\{w_{p_2,1},w_{p_2,2}\}=\{d,e\}$. Moreover, $\epsilon(p_1)=\epsilon(p_2) \cdot (-1)^{\nu_{p_1,p_2}^- +1}$, where $\nu_{p_1,p_2}^-$ is the number of $i$'s with $\nu(i)=-1$. Since $\epsilon(p_1)=\epsilon(p_2)$, it follows that $\nu_{p_1,p_2}=1$. These imply that exactly one of the following holds:
\begin{enumerate}[(i)]
\item $b \equiv d \mod a$ and $c \equiv -e \mod a$. ($\sigma(1)=1$, $\sigma(2)=2$, $\nu(1)=1$, $\nu(2)=-1$)
\item $b \equiv e \mod a$ and $c \equiv -d \mod a$. ($\sigma(1)=2$, $\sigma(2)=1$, $\nu(1)=1$, $\nu(2)=-1$)
\item $b \equiv -d \mod a$ and $c \equiv e \mod a$. ($\sigma(1)=1$, $\sigma(2)=2$, $\nu(1)=-1$, $\nu(2)=1$)
\item $b \equiv -e \mod a$ and $c \equiv d \mod a$. ($\sigma(1)=2$, $\sigma(2)=1$, $\nu(1)=-1$, $\nu(2)=1$)
\end{enumerate}

Suppose that $b \equiv d \mod a$ and $c \equiv -e \mod a$. This implies that $b=d$ and $c+e=a$. The fixed point data of $M$ is then
\begin{center}
$\{+,c+e,b,c\}$, $\{+,c+e,b,e\}$, $\{-,b,b,f\}$, $\{-,c,e,f\}$.
\end{center}
By Theorem \ref{t22}, we deduce that
\begin{center}
$\displaystyle 0=\textrm{sign}(M)$

$=\displaystyle (1+2 \sum_{j=1}^{\infty} t^{j(c+e)})(1+2 \sum_{j=1}^{\infty} t^{j b})[(1+2 \sum_{j=1}^{\infty} t^{j c})+(1+2 \sum_{j=1}^{\infty} t^{je})]-(1+2 \sum_{j=1}^{\infty} t^{jf})[(1+2 \sum_{j=1}^{\infty} t^{j b})^2+(1+2 \sum_{j=1}^{\infty} t^{j c})(1+2 \sum_{j=1}^{\infty} t^{je})]$.
\end{center}
That is, we must have
\begin{center}
$\displaystyle (1+2 \sum_{j=1}^{\infty} t^{j(c+e)})(1+2 \sum_{j=1}^{\infty} t^{j b})[(1+2 \sum_{j=1}^{\infty} t^{j c})+(1+2 \sum_{j=1}^{\infty} t^{je})]$

$\displaystyle =(1+2 \sum_{j=1}^{\infty} t^{jf})[(1+2 \sum_{j=1}^{\infty} t^{j b})^2+(1+2 \sum_{j=1}^{\infty} t^{j c})(1+2 \sum_{j=1}^{\infty} t^{je})]$.
\end{center}
This implies that $b=c$ and $c+e=f$, or $e=b$ and $c+e=f$. Either case belongs to Case (1) of Theorem \ref{t11}.

Suppose that $b \equiv e \mod a$ and $c \equiv -d \mod a$. This implies that $b=e$ and $c+d=a$. The fixed point data of $M$ is then
\begin{center}
$\{+,c+d,b,c\}$, $\{+,c+d,d,b\}$, $\{-,b,d,f\}$, $\{-,c,b,f\}$.
\end{center}
By Theorem \ref{t21},
\begin{center}
$\displaystyle 0=\int_M 1=\sum_{p \in M^{S^1}} \frac{1}{\prod_{i=3}^3 w_{pi}}=\frac{1}{(c+d)bc}+\frac{1}{(c+d)db}-\frac{1}{bdf}-\frac{1}{cbf}$.
\end{center}
Therefore, $f=c+d$. This belongs to Case (1) of Theorem \ref{t11}.

Suppose that $b \equiv -d \mod a$ and $c \equiv e \mod a$. This implies that $b+d=a$ and $c=e$. The fixed point data of $M$ is then
\begin{center}
$\{+,b+d,b,c\}$, $\{+,b+d,d,c\}$, $\{-,b,d,f\}$, $\{-,c,c,f\}$.
\end{center}
By Theorem \ref{t22}, we deduce that
\begin{center}
$\displaystyle 0=\textrm{sign}(M)$

$=\displaystyle (1+2 \sum_{j=1}^{\infty} t^{j(b+d)})(1+2 \sum_{j=1}^{\infty} t^{j c})[(1+2 \sum_{j=1}^{\infty} t^{j b})+(1+2 \sum_{j=1}^{\infty} t^{jd})]-(1+2 \sum_{j=1}^{\infty} t^{jf})[(1+2 \sum_{j=1}^{\infty} t^{j b})(1+2 \sum_{j=1}^{\infty} t^{j d})+(1+2 \sum_{j=1}^{\infty} t^{j c})^2]$.
\end{center}
That is, we must have
\begin{center}
$\displaystyle (1+2 \sum_{j=1}^{\infty} t^{j(b+d)})(1+2 \sum_{j=1}^{\infty} t^{j c})[(1+2 \sum_{j=1}^{\infty} t^{j b})+(1+2 \sum_{j=1}^{\infty} t^{jd})]$

$\displaystyle =(1+2 \sum_{j=1}^{\infty} t^{jf})[(1+2 \sum_{j=1}^{\infty} t^{j b})(1+2 \sum_{j=1}^{\infty} t^{j d})+(1+2 \sum_{j=1}^{\infty} t^{j c})^2]$.
\end{center}
This implies that either $b=c$ and $f=b+d$, or $d=c$ and $f=b+d$. In either case, this belongs to Case (1) of Theorem \ref{t11}.

Suppose that $b \equiv -e \mod a$ and $c \equiv d \mod a$. This implies that $b+e=a$ and $c=d$. The fixed point data of $M$ is then
\begin{center}
$\{+,b+e,b,c\}$, $\{+,b+e,c,e\}$, $\{-,b,c,f\}$, $\{-,c,e,f\}$.
\end{center}
By Theorem \ref{t21},
\begin{center}
$\displaystyle 0=\int_M 1=\sum_{p \in M^{S^1}} \frac{1}{\prod_{i=3}^3 w_{pi}}=\frac{1}{(b+e)bc}+\frac{1}{(b+e)ce}-\frac{1}{bcf}-\frac{1}{cef}$.
\end{center}
Therefore, $f=b+e$ and thus $f=a$. This belongs to Case (1) of Theorem \ref{t11}. \end{proof}

With all of the above, we are ready to prove our main result, Theorem \ref{t11}.

\begin{proof}[\textbf{Proof of Theorem \ref{t11}}]
By Lemma \ref{l41}, exactly one of the figures in Figure \ref{fig1} occurs as a signed labeled multigraph describing $M$ that satisfies the properties in Lemma \ref{l35}. If Figure \ref{fig1-1} describes $M$, the fixed point data is precisely Case (1) of Theorem \ref{t11}. By Lemmas \ref{l51}, \ref{l52}, and \ref{l53}, if any of Figures \ref{fig1-2}, \ref{fig1-3}, or \ref{fig1-4} describes $M$, then the fixed point data of $M$ belongs to Case (1) of Theorem \ref{t11}.

Suppose that Figure \ref{fig1-5} describes $M$. By permuting $p_1$ and $p_2$, by permuting $p_3$ and $p_4$, and by changing the sign of every fixed point if necessary, without loss of generality, we may assume that exactly one of the following cases holds for the largest weight:
\begin{enumerate}
\item $b$ is the largest weight.
\item $a$ is the largest weight.
\end{enumerate}

Since the multigraph Figure \ref{fig1-5} has an edge whose vertices have the same sign, Properties 1 and 2 of Lemma \ref{l35} imply that the largest weight is strictly bigger than $a_2 \in W_+$ where $a_2$ and $W_+$ are as in Lemma \ref{l33}.

Suppose that $b$ is the largest weight. By Property 2 of Lemma \ref{l35}, $p_1$ and $p_3$ lie in the same connected component $F$ of $M^{\mathbb{Z}_b}$. Since $b>a_2$ is the largest weight, $\dim F=2$ or $\dim F=4$. If $\dim F=4$, by Lemma \ref{l54}, the fixed point data of $M$ belongs to Case (1) of Theorem \ref{t11}. Assume that $\dim F=2$. Then by Lemma \ref{l56}, one of the following holds for the fixed point data of $M$:
\begin{enumerate}[(i)]
\item $\{+,a,b,c\},\{-,a,b,c\},\{+,d,e,f\},\{-,d,e,f\}$ for some positive integers $a$, $b$, $c$, $d$, $e$, and $f$.
\item $\{+,a,a+b,a+b+c\}, \{-,a,b,b+c\}, \{+,b,c,a+b\}, \{-,c,b+c,a+b+c\}$. for some positive integers $a$, $b$, and $c$.
\end{enumerate}
The fixed point data of Case (i) is Case (1) of Theorem \ref{t11}, and the fixed point data of Case (ii) is Case (2) of Theorem \ref{t11}.

Suppose that $a$ is the largest weight. By Property 2 of Lemma \ref{l35}, $p_1$ and $p_2$ lie in the same connected component $F$ of $M^{\mathbb{Z}_a}$. Since $a>a_2$ is the largest weight, $\dim F=2$ or $\dim F=4$. If $\dim F=4$, by Lemma \ref{l55}, the fixed point data of $M$ belongs to Case (1) of Theorem \ref{t11}. If $\dim F=2$, by Lemma \ref{l57}, the fixed point data of $M$ belongs to Case (1) of Theorem \ref{t11}. \end{proof}

\section{Comparison with results on different types of manifolds} \label{s6}

In this section, we compare circle actions on different types of 6-dimensional manifolds that have 4 fixed points.

We consider almost complex manifolds. Throughout this section, suppose that a circle action on an almost complex manifold $(M,J)$ preserves the almost complex structure $J$. Then the sign of every weight at a fixed point is well-defined. While for oriented manifolds on any even dimension there exists a circle action with 2 fixed points (a rotation of $S^{2n}$), for almost complex manifolds if the circle acts on a compact almost complex manifold with 2 fixed points, then the dimension of the manifold must be either 2 or 6; see \cite{J1}, \cite{Ko1}, \cite{Ko2}, \cite{M2}, \cite{PT}.
In addition, an almost complex manifold with 3 fixed points only exists in dimension 4, see \cite{J1}; also see \cite{J0}.

We compare circle actions on oriented manifolds and almost complex manifolds, in dimension 6 and when there are 4 fixed points.

\begin{theo} \cite{J3} \label{t61} Let the circle act on a 6-dimensional compact almost complex manifold $M$ with 4 fixed points. Then there exist positive integers $a$, $b$, $c$, and $d$ so that one of the following holds for the weights at the fixed points as complex $S^1$-representations:
\begin{enumerate}[(1)]
\item $\{a,b,c\}$, $\{-a,b-a,c-a\}$, $\{-b,a-b,c-b\}$, $\{-c,a-c,b-c\}$. Here, $a$, $b$, and $c$ are mutually distinct.
\item $\{a,a+b,a+2b\}$, $\{-a,b,a+2b\}$, $\{-a-2b,-b,a\}$, $\{-a-2b,-a-b,-a\}$.
\item $\{1,2,3\}$, $\{-1,1,a\}$, $\{-1,-a,1\}$, $\{-1,-2,-3\}$.
\item $\{-a-b,a,b\}$, $\{-c-d,c,d\}$, $\{-a,-b,a+b\}$, $\{-c,-d,c+d\}$.
\item $\{-3a-b,a,b\}$, $\{-2a-b,3a+b,3a+2b\}$, $\{-a,-a-b,2a+b\}$, $\{-b,-3a-2b,a+b\}$ up to reversing the circle action if necessary.
\item $\{-a-b,2a+b,b\}$, $\{-2a-b,a,b\}$, $\{-b,-2a-b,a+b\}$, $\{-a,-b,2a+b\}$.
\end{enumerate} \end{theo}

For a circle action on a compact almost complex manifold, let $p$ be an isolated fixed point. Suppose that $p$ has weights $\{w_{p1},\cdots,w_{pn}\}$ as complex $S^1$-representations, for some non-zero integers $w_{pi}$'s. If $p$ has exactly $n_p$ negative weights, the sign of $p$ is $(-1)^{n_p}$, and hence the fixed point data of $p$ as real $S^1$-representations is $\{(-1)^{n_p},|w_{p1}|,\cdots,|w_{pn}|\}$. With this understood, we see that Case 1 and Case 5 of Theorem \ref{t61} belong to Case (2) of Theorem \ref{t11}, and all the other cases of Theorem \ref{t61} belong to Case (1) of Theorem \ref{t11}.

Note that in Cases 1, 2, 4, and 5 in Theorem \ref{t61}, there exists a manifold with the fixed point data. Case 3 with $a=2$ or $3$ belongs to Case 1 in Theore \ref{t61}. In Case 3, if $a=4$ or $5$, there are examples; see \cite{A}, \cite{Mc}. Case 6 of Theorem \ref{t61} has a possibility that this might be realized as a blow up $S^2$ in $S^6$ equipped with a rotation having 2 fixed points; see \cite{J3} for discussion on this.

If furthermore $M$ admits a symplectic structure and the circle action is Hamiltonian, only the fixed point data of Cases 1-3 in Theorem \ref{t61} can occur, and in Case 3 only possible values for $a$ are $2$, $3$, $4$, and $5$; see \cite{T}. Therefore, in this case for each fixed point data there exists a manifold.

A natural question is what is a possible fixed point data for an 8-dimensional compact oriented $S^1$-manifold with 4 fixed points. The examples of such a manifold are an equivariant sum (or a disjoint union) of rotations of two 8-spheres, $S^2 \times S^6$, and $S^4 \times S^4$. The fixed point datum for the latter two manifolds belong to the fixed point data of the first example. On the other hand, if the circle acts on an 8-dimensional almost complex manifold with 4 fixed points, all the Chern numbers and the Hirzebruch $\chi_y$-genus agree with those of $S^2 \times S^6$ \cite{J4}. For an 8-dimensional oriented $S^1$-manifold with 4 fixed points, one may ask if the fixed point data is the same as that of a disjoint union of rotations of two 8-spheres.

\section{Converting the fixed point data into the empty collection} \label{s7}

In \cite{J6}, the author showed that if the circle acts on a 6-dimensional compact oriented manifold $M$ with isolated fixed points, there is a systematic way of converting the fixed point data of $M$ into the empty collection by applying a combination of a number of types of operations on it. This is proved by showing that we can successively take equivariant connected sums at fixed points of $M$ with itself, $\mathbb{CP}^3$, and 6-dimensional analogue of the Hirzebruch surfaces (and these with opposite orientations) to a fixed point free action on a compact oriented 6-manifold. In this section, we show how this result applies to the case that $M$ has 4 fixed points.

Consider a circle action on a compact oriented manifold $M$ with a non-empty finite fixed point set. Let $p$ be an (isolated) fixed point. In the introduction, in the decomposition of the tangent space $T_pM=\oplus_{i=1}^n L_i$ to $M$ at $p$, we chose an orientation of each $L_i$ so that $S^1$ acts on $L_i$ with positive weight $w_{pi}$. In this section, we allow any orientation of $L_i$; the circle acts on $L_i$ with weight $w_{pi}$, which is now a (non-zero) integer. Let $\epsilon(p)=+1$ if the orientation on $M$ agrees on the orientation on $\oplus_{i=1}^n L_i$, and let $\epsilon(p)=-1$ otherwise.

For an ordered pair $(a,\{a_1,\cdots,a_n\})$ where $a \in \{-1,+1\}$ and $\{a_1,\cdots,a_n\}$ is a multiset of non-zero elements of $\mathbb{Z}^k$, we define an equivalence relation as follows:
\begin{itemize}
\item $(a,\{a_1,\cdots,a_i,\cdots,a_n\})$ is equivalent to $(-a,\{a_1,\cdots,-a_i,\cdots,a_n\})$.
\end{itemize}

With the equivalence relation, we define the \textbf{fixed point data} of $p$ to be the equivalance class $[\epsilon(p),w_{p1},\cdots,w_{pn}]$ of the ordered pair $(\epsilon(p),\{w_{p1},\cdots,w_{pn}\})$. If all $w_{pi}$ are positive, this agrees with the fixed point data of $p$ in the introduction. The \textbf{fixed point data} of $M$ is a collection of the fixed point data of the fixed points of $M$. 

\begin{theorem} \label{t62} \cite{J6}
Let the circle group $S^1$ act on a 6-dimensional compact oriented manifold $M$ with isolated fixed points. To the fixed point data $\Sigma_M$ of $M$, we can apply a combination of the following operations to convert $\Sigma_M$ to the empty collection.
\begin{enumerate}
\item[(1)] Remove $[+,A,B,C]$ and $[-,A,B,C]$ together.
\item[(2)] Remove $[\pm,A,B,C]$ and $[\mp,C-A,C-B,C]$, and add $[\pm,A,B-A,C-A ]$ and $[\mp, B,B-A,C-B]$, where $0<A<B<C$.
\item[(3)] Remove $[\pm,A,B,C]$ and $[\pm,A,C-B,C]$, and add $[\pm, C-B,C-A,A]$, $[\pm, C-B,B,A]$, $[\pm, C-B,A-B,A]$, and $[\mp, C-A,A-B,A]$, where $0<A,B<C$ and $A \neq B$.
\item[(4)] Remove $[\pm,A,A,C]$ and $[\pm,A,C-A,C]$, and add $[\pm,C-A,C-2A,A]$, $[\pm,C-A,A,A]$, $[\pm,C-A,A,A]$, $[\mp,C-2A,A,A]$, where $0<A<C$.
\item[(5)] Remove $[\pm, C, A, A]$ and $[\mp, C, C-A, C-A]$, and add $[\pm,C-A,C-2A,A]$, $[\pm,C-A, A, A]$, $[\pm,C-A, A, A]$, $[\mp,C-2A, A, A]$, $[\pm, A, C-2A, C-A]$, $[\mp,A,C-A,C-A]$, $[\mp,A,C-A,C-A]$, $[\mp, C-2A, C-A, C-A]$, where $0<A<C$.
\end{enumerate}
There is a definite procedure that this ends in a finite number of steps.
\end{theorem}

We check Theorem \ref{t62} for the case of 4 fixed points. Let the circle act on a 6-dimensional compact oriented manifold $M$ with 4 fixed points. Suppose that Case (1) of Theorem \ref{t11} holds for the fixed point data of $M$. That is, the fixed point data of $M$ is
\begin{center}
$[+,a,b,c],[-,a,b,c],[+,d,e,f],[-,d,e,f]$
\end{center}
for some positive integers $a$, $b$, $c$, $d$, $e$, and $f$. We apply Operation (1) of Theorem \ref{t62} twice to convert the fixed point data of $M$ to the empty collection, first by removing $[+,a,b,c]$ and $[-,a,b,c]$ and second by removing $[+,d,e,f]$ and $[-,d,e,f]$.

Next, suppose that Case (2) of Theorem \ref{t11} holds for the fixed point data of $M$; the fixed point data of $M$ is
\begin{center}
$[+,a,a+b,a+b+c], [-,a,b,b+c], [+,b,c,a+b], [-,c,b+c,a+b+c]$. 
\end{center}
for some positive integers $a$, $b$, and $c$. We apply Operation (2) of Theorem \ref{t62} to remove $[+,a,a+b,a+b+c]$ and $[-,c,b+c,a+b+c]$, and add $[+,a,b,b+c]$ and $[-,a+b,b,c]$ to have a collection
\begin{center}
$[-,a,b,b+c]$, $[+,b,c,a+b]$, $[+,a,b,b+c]$, $[-,a+b,b,c]$.
\end{center}
Next, we apply Operation (1) of Theorem \ref{t62} twice to the above collection, first to remove $[-,a,b,b+c]$ and $[+,a,b,b+c]$ and second to remove $[+,b,c,a+b]$ and $[-,a+b,b,c]$, to reach the empty collection.

\end{document}